\documentclass[11pt]{amsart}
\usepackage[T1]{fontenc}
\usepackage{lmodern}
\usepackage[english]{babel}
\usepackage[a4paper, margin=1in]{geometry}
\usepackage{amssymb,amsthm,verbatim,latexsym,amsfonts,lscape}
\usepackage[centertags]{amsmath}
\usepackage{graphicx,youngtab}
\usepackage{rotating}
\usepackage{enumitem}
\usepackage{titlesec}
\usepackage{pgfplots}
\usepackage{tikz}
\pgfplotsset{compat=1.18}
\usetikzlibrary{calc,hobby,patterns,datavisualization}
\usepackage{tabularx}
\usepackage{adjustbox}
\usepackage{indentfirst}
\usepackage{times}
\usepackage{fancyvrb}
\usepackage{pdfpages}

\newcommand{\F}{\mathbb{F}}
\newcommand{\Z}{\mathbb{Z}}

\usepackage[numbers,square]{natbib}

\theoremstyle{plain}
\newtheorem{theorem}{Theorem}[section]
\newtheorem{lemma}[theorem]{Lemma}
\newtheorem{corollary}[theorem]{Corollary}
\newtheorem{prop}[theorem]{Proposition}

\theoremstyle{definition}

\newtheorem{ex}[theorem]{Example}

\theoremstyle{remark}
\newtheorem{remark}[theorem]{Remark}

\makeatletter
\def\@settitle{
  \begin{center}%
    \bfseries\large\uppercase{\@title}%
  \end{center}%
}
\makeatother

\renewenvironment{abstract}{
  \vspace{1em}
  \begin{list}{}{
    \setlength{\leftmargin}{1cm} 
    \setlength{\rightmargin}{1cm}
  }
  \item[]
  \noindent{ \scshape Abstract.}\hspace{0.3em}\footnotesize
}{
  \end{list}
  \par\vspace{1em}
}

\titleformat{\section}
  {\normalfont\scshape\center\large}{\thesection.}{1em}{}
\titleformat{\subsection}
  {\normalfont\itshape}{\thesubsection.}{1em}{}

\titleformat{name=\section,numberless}
{\normalfont\scshape\center\large}{}{0em}{}

\title{COMPANION MATRICES, PERMUTATIONS, AND LATTICE IDEALS}
\author{Nsibiet E. Udo$^\ast$ and Praise Adeyemo}\thanks{$^\ast$ Corresponding author}
\date{September 2025}

\usepackage{hyperref}

\begin{document}

\maketitle

\begin{abstract}
This paper investigates a novel connection between reductions of companion matrices associated with certain binomial ideals \(I_n\) of the polynomial ring \(F[x_1,\dots, x_n]\) over a field \(F\), given by  
	\begin{align}	
		I_n=\langle x_2x_3\cdots x_n-x_1, x_1x_3\cdots x_n-x_2,\dots, x_1x_2\cdots x_{n-1}-x_{n}\rangle, \label{ab:n1}
	\end{align}
and permutation matrices. Specifically, for fixed monomial orders, we observe that the reduced companion matrices yield permutation matrices satisfying group-theoretic relations. We explore the implications of these reductions, their connections to lattice ideals, and characterize the groups generated by these transformations.
\end{abstract}

\section*{Introduction}
Binomial ideals in polynomial rings have gained significant attention due to their rich combinatorial and algebraic properties \citep{Herzog2018,Stanley2016}. The study of their associated quotient algebras offers a profound insight into the interaction between linear algebra, commutative algebra, and combinatorics. For a zero-dimensional ideal \(I_n \subseteq R_n = F[x_1, \dots, x_n]\), where \(F\) is an algebraically closed field of characteristic zero, the finite-dimensional algebra \(A_n = R_n / I_n\) is central to effective algebraic geometry. The matrices \(T_i = T_{x_i}\) representing multiplication by the coordinate functions \(x_i\) relative to a monomial basis, known as companion matrices, play a pivotal role in this context.

In this paper, we study a symmetric family of such ideals \eqref{ab:n1}, which were first introduced in \citep[Proposition~2.1]{Udoetal2025} as reduced complete intersection of colength \(c_n = 1 + (n-2)\cdot 2^{\,n-1}\). The defining generators equate each variable with the product of all others, a symmetry that leads to a highly structured quotient algebra \(A_n\). We analyze the zero-dimensional subscheme defined by \(I_n\), showing that the reduced companion matrices \(P_i = T_{x_i}^{\mathrm{red}}\), obtained by removing the row and column corresponding to the constant term, are permutation matrices in the permutation group \(\mathfrak{S}_{c_n - 1}\) (Theorem~\ref{lem:reduced-companion}).

A key contribution of this paper is the exploration of the group \(G_n\) generated by these permutations. We prove that the matrices commute pairwise (Theorem~\ref{thm:commuting-reduced}), and consequently that \(G_n\) is abelian. In Proposition~\ref{prop:structure_Gn}, we provide its explicit structure, showing that it is isomorphic to \(C_2^{n-2} \times C_{2n-4}\). This identification provides a concrete algebraic model for the behaviour of the companion matrices: the algebra they generate is commutative and is, in fact, isomorphic to the group algebra \(F[G_n]\).

These results not only deepen our understanding of these specific binomial ideals but also highlight their intricate interplay with diverse areas of mathematics. We further establish a connection between \(G_n\) and lattice ideals, demonstrating that the quotient \(\mathbb{Z}^n / L_n\) mirrors the structure of \(G_n\), and that the length of the lattice ideal \(I_{L_n}\) coincides with the order of \(G_n\).

The remainder of the paper is organized as follows. Section~\ref{sec:groebner} is devoted to the commutative algebra of \(I_n\): we compute its Gr\"{o}bner bases in lexicographic order (Proposition~\ref{prop:generating-set}), and construct the monomial basis (Corollary~\ref{cor:monomial-basis}) . In Section~\ref{sec:companion}, we introduce the reduced companion matrices, prove that they are permutation matrices, and analyze the order and characteristic polynomial of the reduced matrices (Proposition~\ref{prop:cycle-structure}). Section~\ref{sec:fintegroup} identifies the finite abelian group structure of \( G_n \) (Proposition~\ref{prop:structure_Gn}) and its \(2 \)-adic split automorphism group structure (Theorem~\ref{thm:structure_aut_Gn}). Finally, Section~\ref{sec:lattices} places the construction in the broader context of lattice ideals and related commutative algebras.

\bigskip

\noindent\textbf{Acknowledgements} We would like to thank M. Dolors Magret and Dominic Joyce for their suggestions, and most especially Bal\'{a}zs Szendr\"{o}i for his careful reading of an earlier version of this manuscript and his many helpful comments on aspects of this paper.

\section{The inhomogeneous ideal}
\label{sec:groebner}

Let \(R_n\) be a polynomial ring over a field \( F \). We consider certain square-free binomial ideals \(I_n\triangleright R_n\) defined by \eqref{ab:n1}. These ideals lie in the Hilbert scheme $\text{Hilb}^{c_n}(\mathbb{A}^n)$ with colength $c_n$, where the sequence \(c_n\) corresponds to \citep[A000337]{Sloane2010}, in the OEIS.

\begin{prop}[\cite{Udoetal2025}]\label{prop:colength}
	The number \(c_{n}=1+(n-2)\cdot 2^{\,n-1}\) is the codimension of the zero-dimensional square-free binomial ideal \(I_n \triangleright R_n \) of size $n$.
\end{prop}

We now explore the Gr\"{o}bner basis of our defining ideal \(I_n\) for a fixed monomial order \(\omega\), denoted by \(\mathcal G_n=\text{gb}_{\text{Lex}}(I_n)\). Specifically, we consider the ideal \[J_n = \text{in}_{\text{Lex}} (I_n) \triangleright R_n,\] defined by a lexicographic order with \(x_1 > \cdots > x_n\). As a preliminary step, we begin by identifying a substantial set of elements in \(I_n\), which will subsequently be shown to form a Gr\"{o}bner basis. 

\begin{prop}
	\label{prop:polynomial_in_In}
	The following polynomials are elements in \(I_n\)
	\begin{enumerate}
		\item[(1)] \(x_n^{2n-3}-x_n\).
		\item[(2)] For each \(k\) from \(2\) to \(n-1\):
		\begin{enumerate}
			\item[(i)] \(x_kx_n^{2(n-2)}-x_k\).
			\item[(ii)] \(x_k^2-x_n^2\).
		\end{enumerate}
		\item[(3)] \(x_1-x_2x_3\cdots x_n\).
	\end{enumerate}
\end{prop}

\begin{proof}
	Write \(f_i=x_1\cdots x_{i-1}x_{i+1}\cdots x_n-x_i\), for all \(1\leq i\leq n\), so that \(I_n=\langle f_1,\ldots,f_n\rangle\). Since \(x_1\equiv x_2x_3\cdots x_n\pmod{I_n}\) and \(x_i^2\equiv x_n^2\pmod{I_n}\), for every \(1\leq i\leq n-1\), we obtain from $f_n$ that
	\[
	\begin{aligned}
	x_n
	&\equiv x_1x_2\cdots x_{n-1}\\
	&\equiv
	(x_2x_3\cdots x_n)x_2\cdots x_{n-1}\\
	&=
	x_n\prod_{i=2}^{n-1}x_i^2\\
	&\equiv x_n^{2n-3}\pmod{I_n}.
	\end{aligned}
	\]
	Therefore \(x_n^{2n-3}-x_n\in I_n\), which proves (1).
	
	Similarly, for $2\leq k\leq n-1$, using $f_k$ and
	$x_1\equiv x_2\cdots x_n$ gives
	\[
	\begin{aligned}
	x_k
	&\equiv x_1\prod_{\substack{2\leq i\leq n\\i\neq k}}x_i\\
	&\equiv
	(x_2\cdots x_n)
	\prod_{\substack{2\leq i\leq n\\i\neq k}}x_i\\
	&=
	x_kx_n^2
	\prod_{\substack{2\leq i\leq n-1\\i\neq k}}x_i^2\\
	&\equiv x_kx_n^{2(n-2)}
	\pmod{I_n}.
	\end{aligned}
	\]
	Thus \(x_kx_n^{2(n-2)}-x_k\in I_n\), proving (2)(i). For \(2\leq k\leq n-1\), multiplying \(f_k\) by \(x_k\) and \(f_n\) by \(x_n\) gives \(x_kf_k-x_nf_n=x_n^2-x_k^2\), and hence \(x_k^2-x_n^2\in I_n\). This proves (2)(ii). Finally, \(x_1-x_2x_3\cdots x_n=-f_1\in I_n\), which proves (3). Thus all the polynomials belong to \(I_n\).
\end{proof}

\smallskip

\begin{prop} 
	\label{prop:generating-set}
	For \(n \geq 3\), the set\\ \resizebox{\textwidth}{!}{$\mathcal G_n =\left\{ x_n^{2n-3}-x_n, x_{n-1}x_n^{2(n-2)} - x_{n-1}, x_{n-1}^2- x_n^2,\dots, x_2x_n^{2(n-2)}-x_2, x_2^2-x_n^2, x_1-x_2x_3\cdots x_n \right\}$}\\ is a Gr\"{o}bner basis for \(I_n\) with respect to the lexicographic order \( x_1 > x_2 > \cdots > x_n \). 
\end{prop}
\begin{proof}
	The ideal \( I_n = \langle f_1, f_2, \dots, f_n \rangle \), where \( f_i = x_1x_2 \cdots x_{i-1} x_{i+1} \cdots x_n - x_i \) for \( i = 1, \dots, n \). For each \(f_i\), the leading term \( \operatorname{Lt}(f_i) \) is divisible by the leading term of an appropriate element in \( \mathcal G_n \). 
	
	Consider the S-pair \(S(f_i, f_j)\) of any two elements \(f_i\) and \(f_j\) in \(I_n\), defined by
	\[
	S(f_i, f_j) = \operatorname{LCM}(\operatorname{Lt}(f_i), \operatorname{Lt}(f_j))\left\{\frac{1}{\operatorname{Lt}(f_i)} \cdot f_i - \frac{1}{\operatorname{Lt}(f_j)} \cdot f_j\right\},
	\]
	where \(\operatorname{Lt}(f_i)\) is the leading term of \( f_i \) and \(\operatorname{LCM}\) denotes the least common multiple. Now given any two distinct generators \(f_i\) and \(f_j\) where \(i \neq j\), we have
	
	\noindent \textit{Case 1: \(i, j > 1\).} If \(j \neq i\) and \(i, j > 1\), then	\(\operatorname{LCM}\bigl(\operatorname{Lt}(f_i),\operatorname{Lt}(f_j)\bigr) = x_1\cdots x_n\) and
	\[
	S(f_i,f_j)=(x_1\cdots x_n) \left\{ \frac{1}{x_1x_2 \cdots x_{i-1}x_{i+1}\cdots x_n}f_i - \frac{1}{x_1x_2 \cdots x_{j-1}x_{j+1}\cdots x_n}f_j \right\} 
	\]
	\[
	=x_if_i-x_jf_j= -x_i^2+x_j^2.
	\]
	This is a combination of \(q_i=x_i^2-x_n^2\) and \(q_j=x_j^2-x_n^2\), so it reduces to zero modulo \(\mathcal G_n\).
	
	\noindent \textit{Case 2: \(i =1\) or \(j = 1\).} Assume without loss of generality that \(i = 1\). Then \(\operatorname{LCM}\bigl(\operatorname{Lt}(f_1),\operatorname{Lt}(f_j)\bigr)=- \operatorname{Lt}(f_j)\) and a direct computation gives
	\[
	S(f_1,f_j)
	=(x_2x_3\cdots x_{j-1}x_{j+1}\cdots x_n) f_1 + f_j
	= x_jx_n^{2(n-2)}-x_j,
	\]
	which is exactly \(p_j\in\mathcal G_n\). Thus the S-polynomial reduces to zero.
	
	It remains to verify the $S$-pairs of the elements of $\mathcal G_n$ themselves. For convenience, write
	\begin{align*}
	p_n & =x_n^{2n-3}-x_n,\\
	p_k & =x_kx_n^{2(n-2)}-x_k,\quad 2\leq k\leq n-1,\\
	q_k & =x_k^2-x_n^2,\quad 2\leq k\leq n-1,
	\end{align*}
	and
	\[
	p_1=x_1-x_2x_3\cdots x_n.
	\]	
	For $2\leq k<\ell\leq n-1$,
	\[
	S(p_k,p_\ell)
	=x_\ell p_k-x_kp_\ell=0.
	\]
	Also,
	\[
	S(p_k,p_n)
	=x_np_k-x_kp_n=0.
	\]	
	For $2\leq k<\ell\leq n-1$,
	\[
	S(q_k,q_\ell)
	=x_\ell^2q_k-x_k^2q_\ell
	=x_n^2(x_k^2-x_\ell^2)
	=x_n^2(q_k-q_\ell),
	\]
	which is a linear combination of \(q_k\) and \(q_\ell\), and hence reduces to zero.
	
	For $k\neq\ell$, the leading terms of $q_k$ and $p_\ell$ are relatively prime, so \(S(q_k,p_\ell)\) reduces to zero. Similarly, \(S(q_k,p_n)\) reduces to zero. For the remaining pair $q_k$, $p_k$,
	\[
	\begin{aligned}
	S(q_k,p_k)
	&=x_n^{2(n-2)}q_k-x_kp_k\\
	&=x_k^2-x_n^{2n-2}\\
	&=(x_k^2-x_n^2)
	-x_n(x_n^{2n-3}-x_n)\\
	&=q_k-x_np_n,
	\end{aligned}
	\]
	which is a combination of \(q_k\) and \(p_n\), and therefore reduces to zero with respect to $\mathcal G_n$.
	
	Finally, the leading term of $p_1$ is $x_1$, which is relatively prime to the leading terms of all the other elements of $\mathcal G_n$. Hence \(S(p_1,p)\longrightarrow 0\) for every $p\in \mathcal G_n\setminus\{p_1\}$.
	
	Moreover, each defining generator $f_i$ reduces to zero modulo $\mathcal G_n$. Indeed, \(f_1=-p_1\). For \(2\leq i\leq n-1\), reducing the leading term of \(f_i\) by \(p_1\), and then reducing the resulting monomial using the polynomials \(q_j=x_j^2-x_n^2\), \(2\leq j\leq n-1\), \(j\neq i\), the remainder is exactly \(p_i\), which is in \(\mathcal G_n\). For \(f_n\), reducing its leading term by \(p_1\) and subsequently using \(q_j=x_j^2-x_n^2\), \(2\leq j\leq n-1\), gives \(x_n^{2n-3}-x_n=p_n\in\mathcal G_n\), which then reduces to zero.
	
	Thus \(I_n=\langle \mathcal G_n\rangle\), and every $S$-polynomial of every pair of elements of $\mathcal G_n$ reduces to zero. By Buchberger's criterion, $\mathcal G_n$ is a Gröbner basis for $I_n$.	
\end{proof}

\allowdisplaybreaks

\begin{corollary}\label{cor:monomial-basis} 
	For \(0 \leq b < 2n-4 \), the monomial \(m(b, T)\) defined as \[m(b,T)=x_n^{b}\prod_{i\in T}x_i \] where \(T\subseteq \{2,3,\dots,n-1\} \) is any subset of size \(k\) with \(0\leq k \leq n-2\), and \[ m(b,T)=x_n^{2n-4} \] when \(b=2n-4\) and \(|T|=0\), spans the finite-dimensional quotient ring \(R_n/J_n\).
\end{corollary}
\begin{proof}
	This set of monomials represents all those that cannot be reduced using the generators of the ideal \(J_n\). The number of elements in this basis is \(2^{n-1}\times (n-2)+1 \), corresponding to the choices of \(T\) and \(b\).
\end{proof}

\section{Reduced companion matrices and permutations}
\label{sec:companion}

In this section, we demonstrate that the reduced companion matrices for each point  \(I_n\in\operatorname{Hilb}^{c_n}(\mathbb A^n)\) are permutation matrices. A permutation matrix \(P_\sigma\), where \(\sigma\in \mathfrak S_n\), is an \(n\times n\) matrix obtained by permuting the rows of the $n\times n$ identity matrix, and hence, $P_{\sigma}^T=P_{\sigma}^{-1}$.

\begin{theorem}\label{lem:reduced-companion}
	For the fixed lexicographic monomial order \(x_1>x_2>\cdots>x_n\) and arbitrary \(n\), consider the companion matrices $T_{j= 1,\dots, n}$ associated with the square-free binomial ideal $I_n \triangleright  R_n$. For each \(T_j\), after reducing the matrix by deleting the first row and column, the resulting matrices are permutation matrices \(P_j \in M_{c_n-1}(F) \) for \(j=1, \dots, n\).
\end{theorem}

\begin{proof}	
	From Corollary~\ref{cor:monomial-basis}, the generating set for \( I_n \) with respect to the lexicographic order yields a basis for \( R_n / I_n \) consisting of monomials
	\[
	\mathcal{B} = \left\{
	m(b,T)=x_n^b\prod_{i\in T}x_i:
	T\subseteq\{2,\ldots,n-1\},\
	0\leq b\leq 2n-4
	\right\},
	\]
	with the convention that \(x_n^{2n-4} \) is the only monomial when \(T = \varnothing\) and \( b = 2n-4 \). The element \(m(0,\varnothing)=1\) is the first basis element. Let \(\mathcal{B}'=\mathcal{B}\setminus\{1\}\). By Proposition~\ref{prop:colength}, $\dim R_n/I_n = c_n < \infty $.
	
	For each \( j \;=\; 1,\dots, n \), consider the \( F\)-linear multiplication map \( m_{x_j}: R_n/I_n \rightarrow R_n/I_n \) defined by	
	\[
	m_{x_j}(f) = x_j f\; \bmod{ I_n}.
	\]  In the basis set \(\mathcal{B}\), the action of \( m_{x_j} \) behaves differently depending on whether \(j\) is in \(T\), \(j=1\), \(j=n\), or \(j \notin T \).	The structure of \(I_n\) imposes relations that simplify products \( x_j \cdot m(b, T) \) in a uniform way:	
	
	\noindent \textit{(R1) \;\; Case \(2\le j\le n-1\) and \(j \notin T\).} If \(j \notin T\) and \(j\in\{2,\ldots,n-1\}\), then	
	\[
	x_j \cdot m(b, T) = x_j \cdot x_n^b \prod_{i \in T} x_i = x_n^b \cdot x_j \prod_{i \in T} x_i \equiv x_n^b \prod_{i \in T \cup \{j\}} x_i = m(b, T \cup \{j\}) \bmod{ I_n},
	\] provided \(b\leq2n-5\). 
	
	\noindent \textit{(R2) \;\; Case \(2\le j\le n-1\) and \(j \in T\).} If \(j \in T\) and \(j\in\{2,\ldots,n-1\}\), then	
	\[
	x_jm(b,T)
	=x_n^b x_j^2
	\prod_{i\in T\setminus\{j\}}x_i
	\equiv
	x_n^{b+2}
	\prod_{i\in T\setminus\{j\}}x_i
	\pmod{I_n}
	=m(b+2,T\setminus\{j\}),
	\]
	provided \(b\leq2n-7\). When \(b=2n-6\), this gives \(x_jm(2n-6,T)
	=m(2n-4,T\setminus\{j\})\), which is reduced using the relation \(x_jx_n^{2n-4}-x_j\in I_n\). When \(b=2n-5\), we obtain
	\[
	\begin{aligned}
	x_jm(2n-5,T)
	&=x_n^{2n-5}x_j^2
	\prod_{i\in T\setminus\{j\}}x_i\\
	&\equiv
	x_n^{2n-3}
	\prod_{i\in T\setminus\{j\}}x_i\\
	&\equiv
	x_n
	\prod_{i\in T\setminus\{j\}}x_i\\
	&=m(1,T\setminus\{j\}).
	\end{aligned}
	\]
	
	\noindent \textit{(R3)\;\; Case \( j = n \).} If $j=n$, then
	\[
	x_n\cdot m(b,T)
	=x_n^{b+1}\prod_{i\in T}x_i
	=m(b+1,T)\pmod{I_n},
	\]
	whenever $b+1<2n-4$. If $b=2n-5$ and $T\neq\varnothing$, then
	\[
	x_n m(2n-5,T)
	=x_n^{2n-4}\prod_{i\in T}x_i
	\equiv
	\prod_{i\in T}x_i
	=m(0,T)\pmod{I_n},
	\]
	using \(x_i x_n^{2n-4}-x_i\in I_n\), for all \(2\leq i\leq n-1\). If $T=\varnothing$, then \(x_nx_n^{2n-5}=x_n^{2n-4}\), which is the exceptional basis element. Finally,
	\[
	x_nx_n^{2n-4}
	=x_n^{2n-3}
	\equiv x_n\pmod{I_n}.
	\]
	Thus multiplication by $x_n$ cyclically permutes the basis elements in $\mathcal{B}'$ corresponding to each fixed $T$.
	
	\noindent \textit{(R4)\;\; Case \( j = 1 \).} Since $1\notin T$, this case is not covered by (R1). Using \(x_1\equiv x_2x_3\cdots x_n\pmod{I_n}\), we obtain
	\[
	\begin{aligned}
	x_1m(b,T)
	&\equiv
	x_n^{b+1}
	\left(\prod_{i=2}^{n-1}x_i\right)
	\left(\prod_{i\in T}x_i\right)\\
	&\equiv
	x_n^{b+1+2|T|}
	\prod_{i\in S\setminus T}x_i
	\pmod{I_n},
	\end{aligned}
	\]
	where \(S=\{2,\ldots,n-1\}\) and the second congruence follows from \(x_i^2\equiv x_n^2\pmod{I_n}\), for all \(2\leq i\leq n-1\). After reducing the power of $x_n$ by \(x_n^{2n-3}\equiv x_n\pmod{I_n}\), the result is again a basis element of $\mathcal{B}'$. Again, observe that \(T\) and \(S\setminus T\) are distinct, so this action cannot fix any basis element. It remains to prove that these actions are bijections on \(\mathcal B'\).
	
	From  \(x_n^{2n-3}-x_n =x_n(x_n^{2n-4}-1)\in I_n\), we see that multiplication by $x_n$ is not injective on the full basis $\mathcal B$, since
	\[
	x_n\cdot x_n^{2n-4}\equiv x_n
	=x_n\cdot1\pmod{I_n}.
	\]
	Thus it is necessary to remove the constant basis element and work on \(\mathcal{B}'=\mathcal{B}\setminus\{1\}\). Thus, $m_{x_j}$ maps $\mathcal{B}'$ into $\mathcal B'$ for every $j=1,\ldots,n$.
	
	For $j=n$, the relations in (R3) show directly that multiplication by $x_n$ permutes the elements of $\mathcal{B}'$. For $2\leq j\leq n-1$, Proposition~\ref{prop:generating-set} gives \(x_j^2\equiv x_n^2\pmod{I_n}\). Hence, as multiplication operators on $R_n/I_n$, \(m_{x_j}^2=m_{x_n}^2\). Since $m_{x_n}$ is a bijection on $\mathcal B'$, $m_{x_n}^2$ is bijective. Consequently, $m_{x_j}^2$ is bijective, and therefore $m_{x_j}$ is injective on the finite set $\mathcal B'$. Hence $m_{x_j}$ is bijective.
	
	The same argument applies to $j=1$, since \(x_1^2\equiv x_n^2\pmod{I_n}\). Therefore multiplication by every $x_j$, $1\leq j\leq n$, permutes the elements of $\mathcal{B}'$.
	
	The matrix representation of the $F$-linear map $m_{x_j}$ with respect to $\mathcal B'$ is therefore a permutation matrix \(P_j\in M_{c_n-1}(F)\). This matrix corresponds to the companion matrix $T_j$ with the first row and column, corresponding to the constant basis element $1$, deleted. Hence, the reduced companion matrices $P_j$ are permutation matrices.
\end{proof}

\begin{ex} \label{eq:permutation-matrices}
	For \(n = 3\), consider the lexicographic order with \(x_1 > x_2 > x_3\). The ideal \(I_3\) is generated by
	\[
	f_1 = x_2x_3 - x_1, \quad f_2 = x_1x_3 - x_2, \quad f_3 = x_1x_2 - x_3.
	\]
	The basis \(\mathcal{B}\) for \(R_3 / I_3\) consists of monomials \(m(b, T) = \displaystyle x_3^b \prod_{i \in T} x_i\), where \(T \subseteq \{2, 3, \dots, n-1\} = \{2\}\) and \(0 \leq b \leq 2\). So \(T\) can be either \(\varnothing\) or \(\{2\} \). If \(T = \varnothing\), then \(m(b,\varnothing) = x_3^b\) for \(0 \leq b \leq 2n-4 = 2\), and so \(m(0,\varnothing) = x_3^0 = 1 \), \(m(1,\varnothing) = x_3^1 = x_3\) and \(m(2,\varnothing) = x_3^2\). Similarly, if \(T =\{2\}\), then \( m(b,\{2\}) = x_3^b \cdot x_2 \) for \(0 \leq b < 2n-4 = 2\), and so \( m(0,\{2\}) = x_3^0 \cdot x_2 = x_2\), \(m(1,\{2\}) = x_3^1 \cdot x_2 = x_2x_3\).
	 
	 Explicitly, the list of \( m(b, T) \) for \( n = 3 \) is
	\[
	m(0,\varnothing) = 1, \quad m(1,\varnothing) = x_3, \quad m(2,\varnothing) = x_3^2,\quad m(0, \{2\}) = x_2, \quad m(1, \{2\}) = x_2x_3.
	\]
	Therefore, 
	\[
	\mathcal{B} = \{1, x_3, x_2, x_2x_3, x_3^2\}.
	\]
	The multiplication maps \(m_{x_j}\) act (on each \( m(b, T) \) in \(\mathcal{B}\)) as follows
	
	\noindent \textit{For \(x_1 \)-action:}
		\begin{align*}
			m_{x_1}(1) & = x_1 \equiv x_2x_3 \mod I_3\\
		 	m_{x_1}(x_3) & = x_1x_3 \equiv x_2 \mod I_3 \\
		 	m_{x_1}(x_2) & = x_1x_2 \equiv x_3 \mod I_3 \\
		 	m_{x_1}(x_2x_3) & = x_1x_2x_3 \equiv x_1^2 \equiv x_3^2 \mod I_3 \\
		 	m_{x_1}(x_3^2) & = x_1x_3^2 \equiv x_1x_3x_3 \equiv x_2x_3 \mod I_3 
		\end{align*}
	
	\noindent \textit{For \( x_2 \)-action:}
		\begin{align*}
			m_{x_2}(1) & = x_2 \mod I_3 \\
			m_{x_2}(x_3) & = x_2x_3 \mod I_3 \\
			m_{x_2}(x_2) & = x_2^2 \equiv x_3^2 \mod I_3 \\
			m_{x_2}(x_2x_3) & = x_2^2x_3 \equiv x_2 \cdot x_1 \equiv x_3 \mod I_3 \\
			m_{x_2}(x_3^2) & = x_2x_3^2 \equiv x_1 \cdot x_3 \equiv x_2 \mod I_3
		\end{align*}
	
	\noindent \textit{For \( x_3 \)-action:}
		\begin{align*}
			m_{x_3}(1) & = x_3 \mod I_3\\
			m_{x_3}(x_3) & = x_3^2 \mod I_3\\
			m_{x_3}(x_2) & = x_2x_3 \mod I_3\\
			m_{x_3}(x_2x_3) & = x_2x_3^2 \equiv x_1x_3 \equiv x_2 \mod I_3\\
			m_{x_3}(x_3^2) & = x_3^3 \equiv x_3^2 \cdot x_1x_2 \equiv x_1x_2 \equiv x_3 \mod I_3
		\end{align*}
	We use the convention that the $r$-th row of the matrix representation records the coordinates of the image of the $r$-th basis element. Thus, if \(\mathcal B=(m_1,\ldots,m_{c_n})\) is the ordered basis, then
	\[
	m_{x_j}(m_r)
	=
	\sum_{\ell=1}^{c_n}(T_j)_{r\ell}m_\ell.
	\]
	With this row convention, the displayed matrices below represent the multiplication maps described above.	
	\[
	T_1 =\begin{pmatrix}
	0 & 0 & 0 & 1 & 0\\ 
	0 & 0 & 1 & 0 & 0\\ 
	0 & 1 & 0 & 0 & 0\\ 
	0 & 0 & 0 & 0 & 1\\ 
	0 & 0 & 0 & 1 & 0
	\end{pmatrix}, \quad
	T_2 =\begin{pmatrix}
	0 & 0 & 1 & 0 & 0\\ 
	0 & 0 & 0 & 1 & 0\\ 
	0 & 0 & 0 & 0 & 1\\ 
	0 & 1 & 0 & 0 & 0\\ 
	0 & 0 & 1 & 0 & 0
	\end{pmatrix},\quad 
	T_3 =\begin{pmatrix}
	0 & 1 & 0 & 0 & 0\\ 
	0 & 0 & 0 & 0 & 1\\ 
	0 & 0 & 0 & 1 & 0\\ 
	0 & 0 & 1 & 0 & 0\\ 
	0 & 1 & 0 & 0 & 0
	\end{pmatrix} 
	\] 	Removing the first row and column (corresponding to the constant term 1) yields the permutation matrices 
	\[
	P_1 = \begin{pmatrix}
	0 & 1 & 0 & 0\\ 
	1 & 0 & 0 & 0\\ 
	0 & 0 & 0 & 1\\ 
	0 & 0 & 1 & 0
	\end{pmatrix}, \quad
	P_2 = \begin{pmatrix}
	0 & 0 & 1 & 0\\ 
	0 & 0 & 0 & 1\\ 
	1 & 0 & 0 & 0\\ 
	0 & 1 & 0 & 0
	\end{pmatrix},\quad 
	P_3 = \begin{pmatrix}
	0 & 0 & 0 & 1\\ 
	0 & 0 & 1 & 0\\ 
	0 & 1 & 0 & 0\\ 
	1 & 0 & 0 & 0
	\end{pmatrix} 
	\] 

\end{ex}

\begin{ex} 
	For \( n = 4 \), consider the lexicographic order with \(x_1 > x_2 > x_3 > x_4\). The ideal \(I_4\) is generated by
	\[
	f_1 = x_2x_3x_4 - x_1, \quad f_2 = x_1x_3x_4 - x_2, \quad f_3 = x_1x_2x_4 - x_3, \quad f_4 = x_1x_2x_3 - x_4.
	\]
	The basis \(\mathcal{B}\) for \(R_4 / I_4\) consists of monomials \( m(b, T) = \displaystyle x_4^b \prod_{i \in T} x_i \), \( T \subseteq \{2, 3, \dots, n-1\} = \{2, 3\}\) where \(0\leq b \leq 3\) for every $T\subseteq\{2,3\}$, together with the exceptional basis element \(m(4,\varnothing)=x_4^4\). Thus
	\[
	\mathcal{B} = \{m(b,T):T\subseteq\{2,3\},\ 0\leq b\leq3\}
	\cup\{x_4^4\}.
	\]
	In particular, 
	\[
	\mathcal{B} = \{1, x_4, x_3, x_2, x_2x_3x_4, x_4^2, x_3x_4, x_2x_4, x_2x_3, x_2x_3x_4^2, x_2x_4^3, x_3x_4^3, x_4^3, x_3x_4^2, x_2x_4^2, x_2x_3x_4^3, x_4^4\},
	\] and \(|\mathcal B|=4(4)+1=17=c_4\). This result is consistent with the Macaulay2 output \cite{Grayson2001}. The reduced companion matrix representations of \( m_{x_1}, m_{x_2}, m_{x_3}, m_{x_4} \) with respect to \( \mathcal{B} \) are \( 16 \times 16 \) permutation matrices, which can be realized by implementing the Sage code in Appendix~\ref{sec:appendix-1} for any arbitrary \(n\).	
\end{ex}

Next, we present some preliminary results about the general properties of our matrix representation of the associated ideal.

\begin{prop}
	\label{prop:derangement}
	Let \( P_j \in M_{c_n-1}(F) \) be the permutation matrix corresponding to the multiplication map \( m_{x_j} \) on \( \mathcal{B}' = \mathcal{B} \setminus \{1\} \), and let \( \sigma_j \in \mathfrak{S}_{c_n-1} \) be the associated permutation. Then \( \sigma_j \) is a derangement. 
	
\end{prop}

\begin{proof}
	The permutation matrix \(P_j \in M_{c_n-1}(F)\) represents the action of \(x_j\) on \(\mathcal{B}'\), where each basis element \(m_r \in \mathcal{B}'\) is indexed by a pair \((b, T)\) with \(T \subseteq \{2, \dots, n-1\}\) for every \(0\leq b<2n-4\), except that $(0,\varnothing)$ is omitted, together with the exceptional
	basis element \(m(2n-4,\varnothing)=x_n^{2n-4}\). Each entry \((P_j)_{k,r}\) of the matrix representation of \(\sigma_j \in \mathfrak{S}_{c_n-1}\) is the coefficient of the monomial \(m_r \in \mathcal{B}'\) in \(m_k:=m_{x_j}(m_{r'}) \in \mathcal{B}'\). That is,
	\begin{align*}
	(P_j)_{k,r} & = \text{coefficient of}\; m_r\, \text{in}\; (m_k:=x_j\cdot m_{r^{\prime}} \bmod{ I_n})\\ & = \begin{cases}
	\text{coefficient of}\; m_{\sigma_j(k)},\, & \text{if}\; r=\sigma_j(k)\\ 0, & \text{otherwise}
	\end{cases}\\ (P_j)_{k,r} & = \delta_{r,\sigma_j(k)},
	\end{align*} where \(\displaystyle m_{r} := m(b,T) = x_n^{b} \prod_{i \in T} x_i\), \(1 \leq r, k \leq c_n-1\), \(r = b \times 2^{n-2} + t\) and \(\displaystyle t = \sum_{i=1}^{n-2} t_i2^{i-1}\) is a unique number in the range \(0 \leq t \leq 2^{n-2}-1\) such that for each \(a_i \in \{2 < 3 < \cdots < n-1\}\), \(t_i = \begin{cases}
	1, \; & \text{if}\; a_i\in T\\ 0, \; & \text{if}\; a_i \notin T
	\end{cases}\). There is a bijection between the index \(r\) and the parameter \((b,T)\) of the basis elements. Put \(M=2^{n-2}\). For $b=0$ and $1\leq t\leq M-1$, assign the index \(r=t\). For \(1\leq b\leq2n-5\) and \(0\leq t\leq M-1\), assign \(r=bM+t\). Finally, assign \(r=(2n-4)M=c_n-1\) to the exceptional basis element $x_n^{2n-4}$. Thus the indices run precisely through \(1,\ldots,c_n-1\). The trace of \(P_j\) is the sum of the diagonal entries \((P_j)_{k,k}\), which correspond to fixed points of \(\sigma_j\).
	
	We show that \((P_j)_{k,k} = 0\) for all \(k\). Suppose the source element is \(m_{r'} = m(b',T')\). Consider the action of \(x_j\) on a basis element \(m_{r'} = m(b',T') \in \mathcal{B}'\).
	
	\noindent {\textit{Case I}: \(2\le j\le n-1\) and \(j \notin T'\), \(j \neq n\).} Then
	\[ 
	m_k := m_{x_j}(m(b',T'))= x_n^{b'} \prod_{i \in T} x_i = m(b', T) \bmod{ I_n}, 
	\] where \(T := T'\cup \{j\}\). Clearly, \(t> t'\), since \(T'\subsetneq T \). Therefore, the index \(k\) of the image is strictly greater than \(r'\), so \(k\neq r'\), and thus \((P_j)_{k,k} = 0\) for all \(k\).
	
	\noindent {\textit{Case II}: \(2\le j\le n-1\) and \(j \in T'\).}  Then
	\[
	m_k := x_n^{b'} \cdot x_j^2 \prod_{i \in T''} x_i = m(b' + 2, T'\setminus \{j\}) \bmod{ I_n}, 
	\] whenever $b'+2<2n-4$. At \(b'=2n-6\) and \(b'=2n-5\), the images reduce to their equivalent forms as in (R2). In every case the subset changes from $T'$ to $T'':= T'\setminus\{j\}$, and so the pair \((b',T')\) differs from \((b'',T'')\), where \(b''\)  varies according to the case of \(b'\). Hence, the index \( k \) of the image differs from the source \(r'\). This implies that \((P_j)_{k,k} = 0\) for all \(k\).
	
	\noindent {\textit{Case III}: \(j = n\).} If \(b'<2n-5\), then
	\[
	m_k := x_n^{b'+1} \prod_{i \in T'} x_i = m(b' + 1, T') \bmod{ I_n},
	\] so the value \(b'\) increases, implying that \(k \neq r'\). Similarly, if $b'=2n-5$ and $T'\neq\varnothing$, then
	\[
	m_k = m_{x_n}(m(2n-5,T'))
	=
	m(0,T')\pmod{I_n},
	\] using \(x_i x_n^{2n-4}-x_i\in I_n\). If $b'=2n-5$ and $T'=\varnothing$, the image is the exceptional basis
	element $x_n^{2n-4}$. Finally,
	\[
	m_k = m_{x_n}(x_n^{2n-4})
	=
	x_n^{2n-3}
	\equiv x_n = m(1,\varnothing)
	\pmod{I_n}.
	\]
	Thus in every case \(k\neq r'\), and this ensures that \( (P_j)_{k,k} = 0 \) for all \(k\).
	
	\noindent {\textit{Case IV}: \( j = 1 \).} Then
	\[
	m_k := m_{x_1}(m(b',T'))
	\equiv
	x_n^{b'+1+2|T'|}
	\prod_{i\in S\setminus T'}x_i
	\pmod{I_n},
	\]
	where \(S=\{2,\ldots,n-1\}\). Thus the subset $T'$ is replaced by its complement \(S\setminus T'\). Since \(S\) is nonempty (as \(n\ge3\)), a finite set cannot equal its own complement, so the resulting basis element is different from \(m(b',T')\). Therefore \((P_1)_{k,k}=0\).
	
	In all cases, \(m_{x_j}(m_{r'}) \neq m_{r'}\), meaning that \(\sigma_j(r') \neq r'\) for all \( r' \). Hence, \(\sigma_j\) has no fixed points, and is therefore a derangement.
\end{proof}

\begin{prop}
	\label{prop:cycle-structure} 
	 Let \( P_j \) be the permutation matrix for \( \sigma_j \in \mathfrak{S}_{c_n-1} \), associated with the ideal \( I_n \). Then the permutation \( \sigma_j \) decomposes into exactly \( 2^{n-2} \) disjoint cycles, each of length \( 2(n-2) \). Consequently, the order and characteristic polynomial of \( P_j \) are given by \[\operatorname{Ord}(P_j) = 2 (n - 2) \quad \text{and} \quad \chi_{P_j}(t) = \left( t^{2 (n-2)} - 1 \right)^{2^{n-2}}.\] 	
\end{prop}

We will proceed with the proof of this result using the group-theoretic structure point of view. However, we now introduce our final result in this section by first giving the argument for invariance of the span \(W\).

\begin{lemma}\label{lem:nonconstant-invariant}
	Let $A_n=R_n/I_n$ and let $\mathcal{B}$ be the lexicographic monomial basis described in Corollary~\ref{cor:monomial-basis}, with $\mathcal{B}':= \mathcal{B}\setminus\{1\}$. For each $1\le j\le n$, multiplication by $x_j$ maps $\mathcal{B}'$ into $\mathcal{B}'$. i.e.,
	\[
	m_{x_j}(\mathcal{B}')\subseteq \mathcal{B}'.
	\]
	Equivalently, for every nonconstant basis element $m(b,T)\neq 1$, the normal form of $x_j\cdot m(b,T)$ is also nonconstant. Hence $W:=\mathrm{span}_F(\mathcal{B}')$ is a common invariant subspace of $A_n$ for all multiplication operators $m_{x_j}$.
\end{lemma}

\begin{proof}
	Because \(x_n^{2n-3}-x_n = 0\), $x_j$ is a zero-divisor. Since zero-divisors cannot be units, there is no element $m(b,T)\in \mathcal B'$ such that \(x_jm(b,T)=1\) in $A_n=R_n/I_n$. Hence \(W\) is a common invariant subspace of $A_n$ for all multiplication operators $m_{x_j}$.
\end{proof}

\begin{theorem}
	\label{thm:commuting-reduced}
	 For all $1\le i,j\le n$, the reduced companion matrices $P_i$ and $P_j$ in \( M_{c_n-1}(F)\) commute. That is,
	\[
	[P_i,P_j]=0 .
	\]
\end{theorem}

\begin{proof}
	Let $A_n=R_n/I_n$ and, for each $k$, let $m_{x_k}:A_n\to A_n$ denote multiplication by $x_k$.
	Because $A_n$ is a commutative $F$-algebra, the operators commute on the whole space. That is, for any \(i, j\) and for any element \(a \in A_n\),
	\[
	m_{x_i}(m_{x_j}(a)) = x_i(x_j a) = (x_i x_j)a = (x_j x_i)a = m_{x_j}(m_{x_i}(a)).
	\] 
    So \(m_{x_i}\,m_{x_j} = m_{x_j}\,m_{x_i}\) on \(A_n\). By Lemma~\ref{lem:nonconstant-invariant}, the subspace
	\[
	W := \mathrm{span}_F(\mathcal{B}^\prime)\subset A_n
	\]
	is \emph{common invariant} for all $m_{x_k}$. Choose the ordered basis \(\mathcal{E}=(\,1,\,\mathcal{B}^\prime\,)\) of $A_n$, so that every $m_{x_k}$ has a block upper–triangular matrix
	\[
	   T_k \;=\;
	\begin{bmatrix}
	\alpha_k & \rho_k \\
	0 & B_k
	\end{bmatrix}
	\]
	with respect to \(\mathcal{E}\), where \( \alpha_k \) is the scalar coefficient of 1 in the normal form of \(m_{x_k}(1)=x_k\) in \(A_n\) (since \(x_k\) is a nonconstant basis monomial, \(\alpha_k =0\)), \(\rho_k\) is the row vector of coordinates of \(x_k\) in the basis \(\mathcal{B}^\prime\), and the lower–left block is $0$ because $W$ is $m_{x_k}$–invariant.
	By definition, the \emph{reduced companion matrix} $P_k$ is the matrix of the restriction $m_{x_k}\!_{\restriction_W}$ in the basis $\mathcal{B}^\prime$, so $P_k=B_k$.
	
	Since $T_iT_j=T_jT_i$ (the $m_{x_k}$ commutes on $A_n$), the $(2,2)$-block of the product gives
	\[
	B_iB_j = B_jB_i .
	\]
	Therefore $P_iP_j=P_jP_i$, i.e.\ $[P_i,P_j]=0$.
\end{proof}

\section{The finite abelian group structure}
\label{sec:fintegroup}
The above leads us in a natural way to the group of transformations generated by these permutation matrices, which form a subgroup of \( \mathfrak{S}_{c_n-1} \). Suppose there are elements \( g_1, \dots, g_n \) of some group \( G_n \triangleright \mathfrak{S}_{c_n-1} \) derived from the companion matrices \( T_1, \dots, T_n \) associated with the ideal \(I_n\). Since the companion matrices commute, these group elements must satisfy the following relations

\begin{align}
g_ig_j & = g_jg_i
\end{align}
Additionally, due to the generating relations of the ideal, we have
\begin{align}
g_1g_2\cdots g_{n-1} & = g_n \nonumber\\ g_1g_2\cdots g_{n-2}g_n & = g_{n-1} \\ \vdots & \nonumber\\ g_2g_3\cdots g_n & = g_1 \nonumber 
\end{align} These two relations lead to a concrete algebraic problem describing the structure of the group \( G_n \).

Define \( G_n \) as \( G_n = \langle g_1, g_2, \dots, g_n \mid g_i g_j = g_j g_i, \; g_1 g_2 \cdots g_{k-1} g_{k+1} \cdots g_n = g_k \text{ for } k=1,\dots,n \rangle \) which forms an abelian group.

\begin{prop}
	\label{prop:structure_Gn}
	The abelian group \(G_n\) defined above is isomorphic to \(C_2^{n-2} \times C_{2n-4}\), where \( C_2 \) and \(C_{2n-4}\) denote cyclic groups of orders 2 and \(2n-4\), respectively.
\end{prop}

\begin{proof}
	From the presentation of the abelian group 
	\[
	G_n = \left\langle g_1, g_2, \dots, g_n \,\Bigg|\, g_i g_j = g_j g_i, \;\; \prod_{\substack{j=1 \\ j \neq k}}^n g_j = g_k \text{ for } k=1,\dots,n \right\rangle.
	\]
	we have that	
	\begin{align}
	\label{eq:s4n5}
	g_1 g_2 \cdots g_{n} = g_{1}^2 = \cdots = g_n^2. & 
	\end{align}
	 Write \( h_i = g_ig_n^{-1} \) for \( i = 1, \dots, n-1\). Then \( h_i^2 = 1 \). Since \(g_n^2 = g_1g_2\cdots g_{n}\), and by iteratively multiplying the equations in \eqref{eq:s4n5} above, this implies
	\[
	\begin{array}{ll}
	(g_1\cdots g_n)^n = (g_1^2 g_2^2 \cdots g_{n}^2). & \\ g_n^{2n-4} = 1, &	
	\end{array}
	\] The group \( G_n \) is now generated by \(h_1, \dots, h_{n-1}\) and \(g_n\), with the relations
	\[
	\begin{array}{ll}
	h_i^2 = 1 \; \text{for}\; i= 1, \dots, n-1, & \\
	h_1h_2\cdots h_{n-1}g_n^{n-2} = 1, & \\	
	g_n^{2n-4} = 1. & 	
	\end{array}
	\] This defines a subgroup of \( C_2^{n-1} \times C_{2n-4} \). Using the relation \( h_1h_2\cdots h_{n-1}g_n^{n-2} = 1 \), we can express \( h_{n-1}\) as	
	\[h_{n-1} = h_1 h_2 \cdots h_{n-2} g_n^{n-2}.\]
	This now eliminates the dependent generators. Thus, the group \( G_n\) is generated by \( h_1, h_2, \dots, h_{n-2} \) all satisfying \(h_i^2 = 1 \), forming a free abelian group isomorphic to \( C_2^{n-2} \), and \( g_n \) which satisfies \( g_n^{2n-4} = 1 \), generating a cyclic group isomorphic to \( C_{2n-4} \). The relations show that \(G_n\) can be written as \[ G_n = \langle h_1, \dots, h_{n-2} \mid h_i^2 = 1 \rangle \times \langle g_n \mid g_n^{2n-4} = 1 \rangle. \] Since there are no futher dependencies between these generators \( h_1, \dots, h_{n-2} \), \( g_n \), \( G_n\) is isomorphic to the direct product \(C_2^{n-2} \times C_{2n-4}\).
	
\end{proof}

\begin{ex}
	For \(n = 3\), we obtain
	\[ G_3 = \left\langle g_1, g_2, g_3\, \mid g_ig_j=g_jg_i, g_1g_2=g_3, g_1g_3=g_2, g_2g_3=g_1 \right\rangle. \] From this presentation, we deduce that \( g_3 =g_1g_2,  g_2^2 = g_1^2 = e \), where \( e \) denotes the identity element. Thus, \( G_3 \) is isomorphic to \( C_2 \times C_2 \), the direct product of two cyclic groups of order 2, with \(g_1\) and \(g_2\) as independent generators. Moreover, \(g_1g_2=g_3\) represents the last nontrivial element.\\\\ Notice that this group is precisely \(V_4\), a subgroup of \(\mathfrak{S}_{4}\), generated by the three involutions \( (12)(34) \), \( (13)(24) \) and \( (14)(23) \). These involutions correspond to the three matrices \(P_1\), \(P_2\), and \( P_3 \) in Example \ref{eq:permutation-matrices}.
\end{ex}

We now describe the structural decomposition of the automorphism group of the finite abelian group {\(\text{Aut}(G_n)\)}, where the case \(\text{Aut}(G_3)\) is the classical action of \( \mathfrak{S}_3 \) on the Klein four-group \( C_2 \times C_2 \).

\begin{theorem}
\label{thm:structure_aut_Gn}
	For integers \(n \geq 3\). Let
\[
G_n \cong C_2^{\,r}\oplus C_{2n-4},\qquad r:=n-2,
\]
and write $2n-4=2^{s}m$ with $m$ odd ($s=\nu_2(2n-4)$).

\smallskip
\noindent\textbf{(a) Odd $n$ ($s=1$).}
Then the $2$-primary part of $G_n$ is elementary abelian: $G_n^{(2)}\cong C_2^{\,r+1}$ and
\[
\operatorname{Aut}(G_n)\cong \mathrm{GL}(r\!+\!1,2)\times (\Z/(2n-4)\Z)^\times,
\qquad
|\operatorname{Aut}(G_n)|=|\mathrm{GL}(r\!+\!1,2)|\;\phi(2n-4).
\]

\smallskip
\noindent\textbf{(b) Even $n$ ($s\ge2$).}
Write $C_{2n-4}\cong C_{2^{s}}\oplus C_m$ and let $e$ be the unique element of order $2$ in $C_{2^{s}}$.
Then
\[
\operatorname{Aut}(G_n)\cong
\underbrace{\operatorname{Hom}(C_{2^{s}},\,C_2^{\,r})}_{\cong\,C_2^{\,r}}
\rtimes
\Big(
\underbrace{\operatorname{Hom}(C_2^{\,r},\langle e\rangle)\rtimes \mathrm{GL}(r,2)}_{\cong\,C_2^{\,r}\rtimes \mathrm{GL}(r,2)}
\times (\Z/(2n-4)\Z)^\times
\Big),
\]
and
\[
|\operatorname{Aut}(G_n)|=2^{\,2r}\,|\mathrm{GL}(r,2)|\;\phi(2n-4).
\]

\end{theorem}

\begin{proof}
Split $C_{2n-4}\cong C_{2^{s}}\oplus C_m$ with $\gcd(2^{s},m)=1$.
No nontrivial homomorphisms exist between the odd part $C_m$ and the $2$-primary part of $G_n$, hence by \citep[Lemma~2.1]{Hillar-Rhea} we have
\[
\operatorname{Aut}(G_n)\cong \operatorname{Aut}\big(G_n^{(2)}\big)\times \operatorname{Aut}(C_m)\cong \operatorname{Aut}\big(G_n^{(2)}\big)\times (\Z/m\Z)^\times.
\]
Thus it suffices to compute $\operatorname{Aut}\big(G_n^{(2)}\big)$.

\smallskip
\emph{Case $s=1$.}
Then $G_n^{(2)}=C_2^{\,r}\oplus C_2\cong C_2^{\,r+1}$, an $(r\!+\!1)$-dimensional $\F_2$-vector space.
Group automorphisms are precisely $\F_2$-linear automorphisms:
$\operatorname{Aut}\big(G_n^{(2)}\big)\cong \mathrm{GL}(r\!+\!1,2)$.
Multiplying by $(\Z/m\Z)^\times\cong (\Z/(2n-4)\Z)^\times$ gives the desired structure and order.

\smallskip
\emph{Case $s\ge2$.}
Now $G_n^{(2)}=C_2^{\,r}\oplus C_{2^{s}}$.
Let $e\in C_{2^{s}}$ be the unique element of order $2$ (the $2^{s-1}$-st power of a generator).
It is characteristic, hence fixed by any automorphism.
On the $2$-torsion layer
\[
\Omega_1\big(G_n^{(2)}\big)=C_2^{\,r}\oplus \langle e\rangle\cong \F_2^{\,r}\oplus \F_2,
\]
the induced linear action must fix the line $\langle e\rangle$, so the image of the restriction map
$\operatorname{Aut}\big(G_n^{(2)}\big)\twoheadrightarrow \operatorname{Aut}\big(\Omega_1(G_n^{(2)})\big)$
is the stabilizer of $\langle e\rangle$ in $\mathrm{GL}(r\!+\!1,2)$, which is
\[
\operatorname{Stab}(\langle e\rangle)\cong \operatorname{Hom}(C_2^{\,r},\langle e\rangle)\rtimes \mathrm{GL}(r,2)\cong C_2^{\,r}\rtimes \mathrm{GL}(r,2).
\]
Its kernel consists of automorphisms acting trivially on $\Omega_1$, i.e. “shears’’
$f\in\operatorname{Hom}(C_{2^{s}},C_2^{\,r})\cong C_2^{\,r}$ sending $(v,c)\mapsto (v+f(c),c)$.
This kernel is normal, and units in $(\Z/2^{s}\Z)^\times$ act by precomposition.
Therefore
\[
\operatorname{Aut}\big(G_n^{(2)}\big)\cong
\operatorname{Hom}(C_{2^{s}},C_2^{\,r})\rtimes\big(\operatorname{Hom}(C_2^{\,r},\langle e\rangle)\rtimes \mathrm{GL}(r,2)\big),
\]
and restoring the odd part yields the stated semidirect product with $(\Z/(2n-4)\Z)^\times$.
Taking orders gives $|\operatorname{Aut}(G_n)|=2^{2r}\,|\mathrm{GL}(r,2)|\,\phi(2n-4)$.
\end{proof}

\begin{ex}
	For \(3 \leq n\leq 10 \), the orders of $\operatorname{Aut}(G_n)$ are presented in the table below
	
	\begin{adjustbox}{width=\textwidth}
		\begin{tabular}{llll}
			\(n\) & \(G_n\) & \(\text{Aut}(G_n)\)  & \( |\operatorname{Aut}(G_n)| \) \\
			3 & \(C_2 \times C_2\) & \(\mathrm{GL}(2,2) \cong S_3\) &  6 \\
			4 & \(C_2^2 \times C_4\) & \( C_2^{\,2} \rtimes\big(\,(C_2^{\,2}\rtimes S_3)\times (\Z/4\Z)^\times\big) \cong  C_2^3 \rtimes  S_4\) & 192 \\
			5 & \(C_2^3 \times C_6\) & \((\Z/6\Z)^\times\times \mathrm{GL}(4,2)\cong C_2 \times A_8\) & 40320\\
			6 & \(C_2^4 \times C_8 \) & \( C_2^{\,4}\rtimes\big(\,(C_2^{\,4}\rtimes \mathrm{GL}(4,2))\times (\Z/8\Z)^\times\big) \cong C_2 \times (C_2^5 \rtimes (C_2^4 \rtimes A_8))\) & 20643840\\
			7 & \(C_2^5 \times C_{10}\) & \( (\Z/10\Z)^\times\times \mathrm{GL}(6,2)\cong C_4 \times \mathrm{PSL}(6,2)\) & 80634839040\\
			8 & \(C_2^6 \times C_{12}\) & \( C_2^{\,6}\rtimes\big(\,(C_2^{\,6}\rtimes \mathrm{GL}(6,2))\times (\Z/12\Z)^\times\big) \cong C_2 \times (C_2^7 \rtimes (C_2^6 \rtimes \mathrm{PSL}(6,2)))\) & 330280300707840 \\
			9 & \(C_2^7 \times C_{14}\) & \( (\Z/14\Z)^\times\times \mathrm{GL}(8,2) \cong C_6 \times \mathrm{PSL}(8, 2)\) & 32088382615270195200\\
			10 & \(C_2^8 \times C_{16}\) & \(C_2^8 \rtimes (( C_2^8 \rtimes \mathrm{GL}(8, 2) )\times (\Z/16\Z)^\times)\) & 2803925657432463350169600  
			
		\end{tabular}
	\end{adjustbox}

\medskip
\noindent
We already know that over $\F_2$, $\mathrm{GL}(k,2)=\mathrm{SL}(k,2)$ and has trivial center for $k\ge2$, so $\mathrm{GL}(k,2)\cong \mathrm{PSL}(k,2)$. Also $|\,\mathrm{GL}(r,2)\,|=\prod_{j=0}^{r-1}(2^{r}-2^{j})$ and hence $|\,\mathrm{GL}(r\!+\!1,2)\,|=\prod_{j=0}^{r}(2^{r+1}-2^{j})$.

\end{ex}

\section{The lattice ideal and related algebras}
\label{sec:lattices}
In this section, we explore the connections between the group \(G_n\) and lattice ideals. We describe the subgroup \( L_n \) of \( \Z^n \), its associated lattice ideal 
\[
	I_{L_n} = \langle x_i - \prod_{j \neq i}x_j \ (1 \leq i \leq n), x_n^{2n-4} - 1, x_i^2 - x_n^{2 \cdot \text{H}(n-3)} \ (1 \leq i \leq n-1) \rangle, 
\]
and show that the quotient \( \Z^n/ L_n \) corresponds to the group \( G_n \). 

\begin{prop}
	\label{prop:isomorphism-btw-Gn-quotient-Ln}
	The group \( G_n \) is isomorphic to \( \mathbb{Z}^n / L_n \).
\end{prop}

\begin{proof}
	The lattice \( L_n \subset \mathbb{Z}^n \) is generated by the vectors
	\begin{equation}
	    \begin{cases}
	v_i = 2e_i - \sum_{j=1}^n e_j & (1 \leq i \leq n), \\
	v_{n+1} = (2n - 4)e_n, \\
	v_{i+n+1} = 2e_i - 2H(n-3)e_n & (1 \leq i \leq n-1),
	\end{cases} 
    \label{eqn:s5n1}
	\end{equation}
	where \( H(n-3) \) is the Heaviside step function, \( e_1, \dots, e_n\) are canonical basis for \( \Z^n \). Since \( L_n \) has rank \( n \), we may select \( n \) linearly independent generators, such as \( \{v_1, \dots, v_n\} \). These form the columns of the relation matrix
	\[
	R := 2\boldsymbol{I}_n - Q_n \in M_n(\Z),
	\]
	where \( Q_n \) is the all-ones matrix, and \( \boldsymbol{I}_n\) the \( n \times n \) identity. Explicitly,
	\[
	R = 
	\begin{pmatrix}
	1 & -1 & -1 & \cdots & -1 \\
	-1 & 1 & -1 & \cdots & -1 \\
	-1 & -1 & 1 & \cdots & -1 \\
	\vdots & \vdots & \vdots & \ddots & \vdots \\
	-1 & -1 & -1 & \cdots & 1
	\end{pmatrix}.
	\]
	
	The group is isomorphic to the cokernel of the linear map \( R: \mathbb{Z}^n \to \mathbb{Z}^n \), i.e., \( G_n \cong \mathbb{Z}^n / \operatorname{im}(R) \). To determine the structure of \( \mathbb{Z}^n / L_n \), we compute the Smith normal form (SNF) of \( R \). The invariant factors of \( R \) will reveal the group structure and the index \( [\mathbb{Z}^n : L_n] = |\det(R)| \).
	
	Define the all-ones vector \( \mathbf{1} \) and the vector \( f \) as
	\[
	\mathbf{1} = \sum_{i=1}^n e_i = (1,1,\dots,1)^T, \quad f = \sum_{i=2}^n e_i = (0,1,1,\dots,1)^T \in \mathbb{Z}^n,
	\]
	and consider the unimodular matrices
	\[
	U = \boldsymbol{I}_n + f e_1^T, \quad V = \boldsymbol{I}_n + e_1 f^T.
	\]
	These matrices correspond to sequences of elementary row and column operations: subtracting the first row/column from the others, then adding twice the first row/column to the others. One may verify that \( \det(U) = \det(V) = 1 \), since \( e_1^T f = 0 \).
	
	Now compute the product
	\[
	\begin{aligned}
	URV &= (I + f e_1^T)(2I - \mathbf{1}\mathbf{1}^T)(I + e_1 f^T) \\
	&= (I + f e_1^T)\left[2(I + e_1 f^T) - \mathbf{1}\mathbf{1}^T(I + e_1 f^T)\right] \\
	&= (I + f e_1^T)\left[2I + 2e_1 f^T - \mathbf{1}\mathbf{1}^T - \mathbf{1}(e_1^T \mathbf{1}) f^T\right] \\
	&= (I + f e_1^T)\left[2I + (2e_1 - \mathbf{1})f^T - \mathbf{1}\mathbf{1}^T\right].
	\end{aligned}
	\]
	Multiplying out and simplifying using the identities \( e_1^T f = 0 \), \( e_1^T \mathbf{1} = 1 \), and \( f^T \mathbf{1} = n-1 \), we obtain
	\[
	URV = 2I - e_1 e_1^T - 2f f^T = \begin{pmatrix} 1 & \mathbf{0}^T \\ \mathbf{0} & M' \end{pmatrix},
	\]
	where \( M' = -2(Q_{n-1} - \boldsymbol{I}_{n-1}) \).
	
	Now, the matrix \( K = Q_{n-1} - \boldsymbol{I}_{n-1} \) has eigenvalues \( n-2 \) (with eigenvector \( \mathbf{1} \)) and \( -1 \) (with multiplicity \( n-2 \)). Hence,
	\[
	\det(K) = (n-2)(-1)^{n-2}.
	\]
	Since all entries of \( K \) are 0 or 1, the greatest common divisor of its entries is 1, so the first invariant factor of \( K \) is 1. The invariant factors \( d_1 \mid d_2 \mid \cdots \mid d_{n-1} \) must satisfy \( d_1 \cdots d_{n-1} = |\det(K)| = n-2 \), and since \( d_1 = 1 \) and $K$ is an integer matrix of full rank $n-1$, it follows that
	\[
	\operatorname{SNF}(K) = \operatorname{diag}(1, 1, \dots, 1, n-2),
	\]
	with the entry \(1\) occurring \( n-2 \) times. Therefore,
	\[
	\operatorname{SNF}(M') = \operatorname{diag}(2, 2, \dots, 2, 2(n-2)),
	\]
	with \( 2 \) repeated \(n - 2\) times, and combining with the leading 1 gives
	\[
	\operatorname{SNF}(R) = \operatorname{diag}(1, 2, 2, \dots, 2, 2(n-2)).
	\]
	
	Thus,
	\[
	\mathbb{Z}^n / L_n \cong \mathbb Z^n /  \operatorname{im}(R) \cong (\mathbb Z/2 \mathbb Z)^{\,n-2}  \oplus \mathbb{Z}/(2(n-2))\mathbb Z,
	\]
	with \( n-2 \) copies of \( \mathbb{Z}/2 \), and
	\[
	|\mathbb{Z}^n / L_n| = |\det(R)| = 1 \cdot 2^{n-2} \cdot 2(n-2) = (n-2)2^{n-1}.
	\]
	This completes the proof. 
	
\end{proof}

\begin{prop}
	\label{prop:radical_ILn}
	The lattice ideal \(I_{L_n}\) is radical and zero-dimensional. 
\end{prop}

\begin{proof}
	The lattice ideal \(I_{L_n} \subset R_n\) is generated by binomials of the form 
	\begin{alignat}{2}
	\label{eq:main}
	f_i = x_i - \prod_{j \neq i} x_j \quad (1 \leq i \leq n), & & \quad f_{n+1} = x_n^{2n-4} - 1, &\quad f_{i+n+1} = x_i^2 - x_n^{2H(n-3)} \quad (1 \leq i \leq n-1).  \end{alignat} Since \(F\) is algebraically closed of characteristic zero, and \(I_{L_n}\) is generated by binomials without monomial factors, it is radical.
	
	We now show that \(I_{L_n} \subset R_n\) is zero-dimensional, i.e., \(\dim_k R_n / I_{L_n} < \infty\). The variety \(V(I_{L_n})\) consists of points \((x_1, \dots, x_n) \in F^n\) satisfying the system
	\[
	x_i = \prod_{j \neq i} x_j \quad (1 \leq i \leq n), \qquad x_n^{2n-4} = 1, \qquad x_i^2 = x_n^{2H(n-3)} \quad (1 \leq i \leq n-1).
	\]
	
	We consider two cases.
	
	\noindent {\itshape Case 1: \( n = 3 \).}
	Here, \(H(0) = 0\), so \( x_i^2 = 1 \) for \( i = 1, 2, 3 \). Hence, \(x_i = \pm 1\). The relations \(x_1 = x_2 x_3\), \(x_2 = x_1 x_3\), and \(x_3 = x_1 x_2\) imply
	\[
	x_1 x_2 x_3 = x_1^2 = x_2^2 = x_3^2 = 1.
	\]
	Thus, the product \(x_1 x_2 x_3 = 1\), and each \(x_i^2 = 1\). The only solutions are
	\[
	(1,1,1), \quad (1,-1,-1), \quad (-1,1,-1), \quad (-1,-1,1),
	\]
	yielding \( |V(I_{L_3})| = 4 \).
	
	\noindent {\itshape Case 2: \( n > 3 \).} 
	Now \( H(n-3) = 1 \), so \( x_i^2 = x_n^2 \) for \( 1 \leq i \leq n-1 \). Thus, \( x_i = \varepsilon_i x_n \) with \( \varepsilon_i = \pm 1 \). Also, \( x_n^{2n-4} = 1 \), so \( x_n \) is a \((2n-4)\)th root of unity.
	
	From \(x_n = \prod_{j=1}^{n-1} x_j\), we substitute \(x_j = \varepsilon_j x_n\) to obtain
	\[
	x_n = \left( \prod_{j=1}^{n-1} \varepsilon_j \right) x_n^{n-1}.
	\]
	Since \(x_n \neq 0\) (as \(x_n^{2n-4}=1\)), divide both sides by \(x_n\)
	\[
	1 = \left( \prod_{j=1}^{n-1} \varepsilon_j \right) x_n^{n-2}.
	\]
	Let \(\varepsilon = \prod_{j=1}^{n-1} \varepsilon_j \in \{\pm 1 \}\). Then \(\varepsilon = x_n^{-(n-2)}\). Note that \(x_n^{-(n-2)}\) is also a \((2n-4)\)th root of unity, and in fact \((x_n^{-(n-2)})^2 = x_n^{-2(n-2)} = 1\) since \(x_n^{2n-4} = 1\). Hence, \(x_n^{-(n-2)} = \pm 1\).
	
	For each \( x_n \) satisfying \( x_n^{2n-4} = 1 \), the condition \( \varepsilon = x_n^{-(n-2)} \) determines the product of the signs \( \varepsilon_1, \ldots, \varepsilon_{n-1} \). There are \( 2^{n-2} \) choices of signs \( (\varepsilon_1, \ldots, \varepsilon_{n-1}) \in \{ \pm 1 \}^{n-1} \) with fixed product \( \varepsilon \). Since there are \( 2n-4 \) choices for \( x_n \), the total number of solutions is
	\[
	|V(I_{L_n})| = (2n-4) \cdot 2^{n-2} = 2^{n-1}(n-2) = c_n - 1.
	\]	
	
	In both cases, \(|V(I_{L_n})| = c_n -1\), as shown in Proposition~\ref{prop:colength}. Since \(V(I_{L_n})\) has finitely many points and \( I_{L_n} \) is radical, we conclude that \(\dim_F R_n / I_{L_n} = |V(I_{L_n})| = c_n -1 < \infty\). Hence, \(I_{L_n}\) is zero-dimensional.
	
\end{proof}

We now describe how the lattice ideal \(I_{L_n}\) relates to the defining ideal \(I_n\) of the configuration. 

\begin{prop}
	\label{prop:relation_In_ILn}
	Let \(R_n = F[x_1, \dots, x_n]\) and let \(\mathfrak{m} = (x_1, \dots, x_n)\) denote the maximal ideal at the origin. 
	Let \(I_n \subset R_n\) be the square–free binomial ideal defined in \eqref{ab:n1}, and let \(I_{L_n} \subset R_n\) be the lattice ideal generated by the binomials \(f_i\) in \eqref{eq:main}. 
	Then \(I_{L_n}\) is the radical component of \(I_n\) away from the origin, and, in fact,\[
	I_n \;=\; I_{L_n} \;\cap\;  \mathfrak{m}.
	\]
\end{prop}

\begin{proof}
	The variety computations show that the zero set \(V(I_n)\) consists of the origin together with the finite set of points parametrised by the lattice construction. From Proposition~\ref{prop:colength}, we have \(|V(I_n)| = c_n = 1 + (n-2)2^{\,n-1}\) points for \(I_n\), while the lattice ideal \(I_{L_n}\) cuts out the same family of nonzero solutions and satisfies \(|V(I_{L_n})| = c_n - 1\), the origin being excluded. 
	Hence
	\[
	V(I_n) = V(I_{L_n}) \cup \{0\}.
	\]
	Both ideals are generated by binomials without monomial factors and are shown to be radical and zero-dimensional. Therefore, the ideals of these varieties are intersections of their reduced primary components. 
	Since \( V(I_{L_n}) \) is exactly the complement of the origin in \( V(I_n) \), the only additional reduced component of \( I_n \) is the maximal ideal \( \mathfrak{m} \) of the origin. 
	Consequently, the radical primary decomposition of \( I_n \) is
	\[
	I_n = I_{L_n} \cap  \mathfrak{m}.
	\]
	Finally, since both \( I_n \) and \( I_{L_n} \) are already radical, no embedded nilpotent components appear, and hence the equality holds as an equality of ideals, not just radicals
	\[
	I_n = I_{L_n} \cap  \mathfrak{m}.
	\]
	
\end{proof}

\begin{remark}
	(i) The equality is consistent with the length computation 
\(\operatorname{len}(R_n/I_n) = c_n = \operatorname{len}(R_n/I_{L_n}) + 1\), 
the extra \(1\) corresponding to the origin. 
(ii) Equivalently, \(I_{L_n}\) is the saturation of \(I_n\) away from \(\mathfrak{m}\):
\[
I_{L_n} = (I_n :  \mathfrak{m}^{\infty}).
\]
\end{remark}

We now present our main result in this section, following the discussion in \citep[Sect.~7.1, Theorems~7.3--7.4]{Miller-Sturmfels2005}.

\begin{theorem}\label{thm:lattice-algebras}
	Let \(L_n \subseteq \Z^n\) be a lattice generated by the set of vectors in \eqref{eqn:s5n1}.  Let \(I_{L_n} \subset R_n\) be the associated lattice ideal. Then there is a canonical \(F\)-algebra isomorphism
	\[
	R_n / I_{L_n} \cong F[\mathbb{Z}^n / L_n].
	\] In other words, if the associated lattice ideal \(I_{L_n}\) in \(R_n\)  is of finite codimension, then its length is the same as the order of the group \(\Z^n / L_n\). In this case, \[ \text{len} (R_n/I_{L_n}) = | G_n |. \]
\end{theorem}

\begin{proof}
	By Proposition~\ref{prop:isomorphism-btw-Gn-quotient-Ln}, \(G_n \cong \mathbb{Z}^n / L_n\) with independent generators \(h_1, \dots, h_{n-2}\) of order 2 and a generator \( t\) of order \(2n-4\) defined by 
	
	$$
	g_i=h_i\,t\quad(1\le i\le n-2),\qquad g_{n-1}=(h_1\cdots h_{n-2})\,t^{\,n-1},\qquad g_n=t.
	$$ The group algebra $F[\mathbb{Z}^n/L_n]$ is thus the commutative $F$-algebra of dimension \(|\Z^n /L_n| = c_n -1\) generated by $u_{h_1},\dots, u_{h_{n-2}}, u_t$ with the obvious group relations
	
	$$
	u_{h_i}^2=1,\qquad u_t^{\,2n-4}=1,\qquad u_{h_i}u_t=u_t u_{h_i},
	$$
	and all generators commute.
	
	Define a map \(\Phi: R_n \to F[\mathbb{Z}^n/L_n]\) by  
	\[
	\Phi(x_i) = u_{g_i}, \quad \text{for } i = 1, \dots, n,
	\]  
	where \(g_i = \bar{e}_i \in \mathbb{Z}^n / L_n\) is the image of \( e_i \) in \( \mathbb{Z}^n / L_n  \).  Expressed in terms of the chosen generators \(h_j\) and \(t\) the explicit image of each \(x_i\) is 	
	$$
	\Phi(x_i)=u_{h_i}\,u_t \quad(1\le i\le n-2),\qquad \Phi(x_{n-1}) \;=\; u_{h_1}\,u_{h_2}\cdots u_{h_{n-2}}\,u_t^{\,n-1}, \qquad \Phi(x_n)=u_t.
	$$
	
	Extend \(\Phi\) linearly and multiplicatively to all of \(R_n\). This map is a canonical \( F \)-algebra homomorphism by construction: it preserves addition due to linearity, multiplication due to the identity \(u_g u_{g'} = u_{gg'}\) induced by the usual group-algebra convolution, and the unit since \(\Phi(1) = u_1\).
	
	We now verify that \(\Phi\) vanishes on the generators of \(I_{L_n}\), so it descends to a well-defined map \(\overline{\Phi}: R_n / I_{L_n} \to F[\mathbb{Z}^n / L_n]\). 
	
	The defining binomials of $I_{L_n}$ are precisely the relations in the quotient group $\mathbb{Z}^n/L_n$: the relation \(\bar{e}_i = \sum_{j \neq i} \bar{e}_j\) in \(\mathbb{Z}^n / L_n\) implies \(u_{g_i} = \prod_{j \neq i} u_{g_j}\), so \(\Phi(f_i) = 0\); the relation \((2n-4)\bar{e}_n = 0\) implies \(u_t^{2n-4} = 1\), so \(\Phi(f_{n+1}) = 0\); and the relations \(2\bar{e}_i = 2H(n-3)\bar{e}_n\) imply \(u_{g_i}^2 = u_t^{2H(n-3)}\), so \(\Phi(f_{i+n+1}) = 0\). Thus every generator of $I_{L_n}$ maps to zero, so $\Phi(I_{L_n})=0$. Consequently $\Phi$ induces a well-defined homomorphism
	
	$$
	\overline{\Phi}:R_n/I_{L_n}\;\longrightarrow\;F[\mathbb{Z}^n/L_n].
	$$
	
	The induced map is surjective because the \(u_{g_i}\) generate the group algebra. Since \(\mathbb{Z}^n / L_n\) is a finite abelian group, the group algebra \(F[\mathbb{Z}^n / L_n]\) is a finite-dimensional vector space over \(F\). Its dimension is exactly \(|\mathbb{Z}^n / L_n|\). Moreover, because \(I_{L_n}\) is radical and zero dimensional (Proposition~\ref{prop:radical_ILn}), the ring \(R_n/I_{L_n}\) is an Artinian reduced ring, and its length equals its dimension as a \(F\)-vector space. Therefore, \(\overline{\Phi }\) is injective. Hence, \(\overline{\Phi}\) is an isomorphism of \(F\)-algebras, and 
	\[
	\operatorname{len}(R_n / I_{L_n}) = \dim_F (R_n/I_{L_n}) = | \mathbb{Z}^n / L_n | = \dim_F F[\mathbb{Z}^n / L_n] = |G_n|.
	\]
    
\end{proof}

\begin{proof}[Conclusion of the proof of Proposition~\ref{prop:cycle-structure}]
	By Theorem~\ref{thm:lattice-algebras}, we have that the lattice ideal \( I_{L_n} \subset R_n = F[x_1,\dots,x_n] \) admits the canonical identification  
	\[
	R_n / I_{L_n} \;\cong\; F[\mathbb{Z}^n / L_n] \;\cong\; F[G_n],
	\]
	where \( G_n \cong \mathbb{Z}^n / L_n \) is the associated finite abelian group. Under this isomorphism, each monomial class in \( R_n / I_{L_n} \) corresponds naturally to a basis element of the group algebra \( F[G_n] \).  
	In particular, the action of multiplication by the variable \( x_j \) is precisely the action of left multiplication by the group element 
	\[
	g_j := \overline{e_j} \in G_n,
	\]
	the image of the \( j \)-th standard basis vector of \( \mathbb{Z}^n \). Hence, the permutation matrix \( P_j \) representing \( m_{x_j} \) is exactly the matrix of the regular representation of \( g_j \) on \( F[G_n] \).
	
	This identification allows the study of \( P_j \) to be reduced to understanding the left regular action of \( g_j \) on the finite group \( G_n \); the orbit structure of this action corresponds precisely to the coset decomposition of \( G_n \) modulo the cyclic subgroup \( \langle g_j \rangle \).
	
	By Proposition~\ref{prop:structure_Gn}, the group \( G_n \) decomposes as \( G_n \;\cong\; C_2^{\,n-2} \times C_{2n-4}\), with commuting generators that may be chosen so that $g_j=h_j\,t$ for $1\le j\le n-1$, where $h_j$ has order $2$ and $t$ has order $2n-4$. Since these components commute, the order of $g_j$ equals the order of the product $h_j t$, i.e., $\operatorname{ord}(g_j)=\operatorname{lcm}(2,2n-4)=2n-4=2(n-2)$. 
	Left multiplication by \( g_j \) thus partitions \( G_n \) into cosets of the cyclic subgroup  \( \langle g_j \rangle \), each forming an orbit of length $\operatorname{ord}(g_j)=2(n-2)$.  
	As \(|G_n| = (n-2)2^{\,n-1}\) (Proposition~\ref{prop:structure_Gn}), the number of distinct cosets, and hence the number of disjoint cycles in the permutation induced by \( g_j \), is
	\[
	\frac{|G_n|}{\operatorname{ord}(g_j)} = \frac{(n-2)2^{\,n-1}}{2(n-2)} = 2^{\,n-2}.
	\]
	Therefore,  \( \sigma_j \) decomposes into exactly \( 2^{n-2} \) cycles each of length \( 2(n-2) \), giving
	\[
	\operatorname{Ord}(P_j) = 2(n-2)
	\quad\text{and}\quad
	\chi_{P_j}(t) = \big(t^{2(n-2)} - 1\big)^{2^{\,n-2}},
	\]
	as claimed.
\end{proof}

\newpage
\appendix
\renewcommand{\thesection}{A} 
\refstepcounter{section}
\section*{Appendix A: Sage Implementation}
\phantomsection
\label{sec:appendix-1}

\begin{verbatim}
from itertools import product

def monomials_up_to_degree(gens, d):
		"""
		Generate all monomials in gens of total degree <= d.
		"""
		n = len(gens)
		monoms = []
		for exps in product(range(d+1), repeat=n):
				if sum(exps) <= d:
						monoms.append(prod(g**e for g, e in zip(gens, exps)))
		return monoms

def compute_basis(R, G, gens, max_degree=10):
		"""
		Automatically compute a basis for the quotient R/I
		by reducing monomials until the basis stabilizes.
		"""
		B = []
		prev_len = -1
		d = 0
		while d <= max_degree:
				monoms = monomials_up_to_degree(gens, d)
				for m in monoms:
						p = m.reduce(G)
						if p != 0 and p not in B:
								B.append(p)
				if len(B) == prev_len:   # basis did not grow
						break
				prev_len = len(B)
				d += 1
		return B

def companion_matrices(n, max_degree=10):
		"""
		Construct multiplication matrices for the quotient ring
		Q = QQ[x1,...,xn]/I where
		I = (product of all vars except xi - xi).
		
		Automatically finds the basis by increasing degree until stable.
		"""
		# Polynomial ring in n variables
		varnames = [f"x{i+1}" for i in range(n)]
		R = PolynomialRing(QQ, varnames, order='lex')
		gens = R.gens()
		
		# Build ideal I
		Igens = []
		for i in range(n):
				prod_except_i = prod(gens[j] for j in range(n) if j != i)
				Igens.append(prod_except_i - gens[i])
		I = R.ideal(Igens)
		
		# Groebner basis
		G = I.groebner_basis()
		
		# Automatically find basis
		B = compute_basis(R, G, gens, max_degree=max_degree)
		
		print(f"\nBasis B for n={n}:")
		print(B)
		
		# Multiplication matrix builder
		def mult_matrix(v, B, G):
				msize = len(B)
				M = matrix(QQ, msize)
				for j, bj in enumerate(B):
						p = (v*bj).reduce(G)
						coeffs = [p.monomial_coefficient(b) for b in B]
						M[:, j] = vector(coeffs)
				return M
				
		# Compute matrices for all variables
		matrices = {str(g): mult_matrix(g, B, G) for g in gens}
		
		return B, matrices

# === Interactive Run ===
try:
		n = int(input("Enter number of variables n: "))
except:
		n = 3  # fallback default

B, matrices = companion_matrices(n)

for v, M in matrices.items():
		print(f"\nMatrix for {v}:\n{M}")
\end{verbatim}

\vspace{2em}
\noindent
\textsc{University of Calabar, Nigeria} \\
\textit{Email address:} \texttt{nsibietudo@unical.edu.ng}

\vspace{1em}
\noindent
\textsc{University of Ibadan, Nigeria} \\
\textit{Email address:} \texttt{ph.adeyemo@ui.edu.ng}

\end{document}